\documentclass[a4paper,10pt]{amsart}

\usepackage[english]{babel}
\usepackage[utf8]{inputenc}
\usepackage{ulem} 

\usepackage{amssymb}
\usepackage{amsmath}
\usepackage{amsthm}
\usepackage{bbm}
\usepackage{mathtools}

\usepackage{enumerate}
\usepackage{tikz}
\usepackage{marginnote}
\usepackage{mathtools}
\usepackage{hhline}
\usepackage{array}
\usepackage{multirow}
\usepackage{tabularx}

\newcommand{\bbC}{\mathbb{C}}

\newcommand{\bbN}{\mathbb{N}}

\newcommand{\bbR}{\mathbb{R}}

\newcommand{\calA}{\mathcal{A}}
\newcommand{\calB}{\mathcal{B}}

\newcommand{\calL}{\mathcal{L}}

\DeclareMathOperator{\id}{id}
\DeclareMathOperator{\one}{\mathbbm{1}}
\DeclareMathOperator{\re}{Re}

\newcommand{\argument}{\mathord{\,\cdot\,}}
\newcommand{\dx}{\,\mathrm{d}}
\DeclareMathOperator{\sgn}{sgn}

\DeclarePairedDelimiter{\norm}{\lVert}{\rVert}
\DeclarePairedDelimiter{\modulus}{\lvert}{\rvert}

\newcommand{\ue}{\mathrm{e}}
\newcommand{\ui}{\mathrm{i}}
\DeclareMathOperator{\divergence}{div}

\newcommand{\spec}{\sigma}

\newcommand{\pnt}{{\operatorname{pnt}}}

\theoremstyle{definition}
\newtheorem{definition}{Definition}[section]
\newtheorem{remark}[definition]{Remark}

\newtheorem{example}[definition]{Example}

\theoremstyle{plain}
\newtheorem{proposition}[definition]{Proposition}
\newtheorem{lemma}[definition]{Lemma}
\newtheorem{theorem}[definition]{Theorem}
\newtheorem{corollary}[definition]{Corollary}

\numberwithin{equation}{section}

\begin{document}

\normalem

\title[Long-time behaviour of parabolic systems]{Convergence to equilibrium for linear parabolic systems coupled by matrix-valued potentials}
\author{Alexander Dobrick}
\address[A. Dobrick]{Christian-Abrechts-Universit\"at zu Kiel, Arbeitsbereich Analysis, Heinrich-Hecht-Platz 6, 24118 Kiel, Germany}
\email{dobrick@math.uni-kiel.de}
\author{Jochen Gl\"uck}
\address[J. Gl\"uck]{Bergische Universit\"at Wuppertal, Fakult\"at f\"ur Mathematik und Naturwissenschaften, Gaußstr.\ 20, 42119 Wuppertal, Germany}
\email{glueck@uni-wuppertal.de}
\date{\today}
\begin{abstract}
	We consider systems of parabolic linear equations, subject to Neumann boundary conditions on bounded domains in $\bbR^d$, that are coupled by a matrix-valued potential $V$, and investigate under which conditions each solution to such a system converges to an equilibrium as $t \to \infty$.
	
	While this is clearly a fundamental question about systems of parabolic equations, it has been studied, up to now, only under certain positivity assumptions on the potential $V$. Without positivity, Perron--Frobenius theory cannot be applied and the problem is seemingly wide open. In the present article, we address this problem for all potentials that are $\ell^p$-dissipative for some $p \in [1,\infty]$. While the case $p=2$ can be treated by classical Hilbert space methods, the matter becomes more delicate for $p \not= 2$. We solve this problem by employing recent spectral theoretic results that are closely tied to the geometric structure of $L^p$-spaces.
\end{abstract}

\maketitle

\section{Introduction}

\subsection*{Coupled parabolic equations}

On a bounded domain $\Omega \subseteq \bbR^d$ with sufficiently smooth boundary, consider elliptic differential operators $\calA_1, \dots, \calA_N$ in divergence form with Neumann boundary conditions and a bounded measurable function $V \colon \Omega \to \bbC^{N \times N}$. In this paper, we are interested in the long-time behaviour of the solutions to the coupled parabolic equation
\begin{align}
	\label{eq:the-equation}
	\frac{d}{dt}
	\begin{pmatrix}
		u_1 \\ \vdots \\ u_N
	\end{pmatrix}
	= 
	\begin{pmatrix}
		\calA_1 u_1 \\
		\vdots \\
		\calA_N u_N
	\end{pmatrix}
	+ V
	\begin{pmatrix}
		u_1 \\ \vdots \\ u_N
	\end{pmatrix}.
\end{align}
Such equations arise, for instance, as linearizations of reaction-diffusion equations; this renders uniform convergence of the solutions to $0$ a property of~\eqref{eq:the-equation} that is particularly desirable, since this property is related (by the principle of linearized stability, see for instance \cite[Theorem~10.2.2]{Cazenave1998}) to asymptotic stability of equilibria of reaction-diffusion equations.
Asymptotic stability of equilibria has been studied for many specific classes of reaction-diffusion equations and by various methods; see, for instance, \cite[Section~3]{Haraux1988}, \cite[Theorem~3]{Pierre2017}, \cite[Theorem~1.1]{Pierre2018} (and e.g.\ \cite[Theorem~1.4]{Amann1978} for an equation with Robin rather than Neumann boundary conditions).

On the other hand, there appears to be a remarkable gap in the literature concerning the convergence of the solutions to the linear equation~\eqref{eq:the-equation}, in its general form, to possibly non-zero equilibria.
The authors are aware of only three convergence results for equations similar to~\eqref{eq:the-equation}:

\begin{enumerate}[(i)]
	\item 
	In \cite[Sections~7 and~8]{Nagel1989} the long-term behaviour of such systems (though in non-divergence form) was considered 
	with Dirichlet boundary conditions and under assumptions on the potential 
	that ensure positivity of the solution semigroup.  
	The arguments to obtain convergence to equilibrium rely on Perron--Frobenius theory. 
	
	\item 
	In \cite[Theorem~4.4]{Addona2019} parabolic systems on the whole space $\bbR^d$ are studied, 
	also under certain positivity assumption on the potential and by employing arguments from Perron--Frobenius theory.
	
	\item 
	In \cite[Section~5]{Dobrick2021} the present authors used recent spectral theoretic results from \cite{Glueck2016b} to discuss the long-term behaviour of solutions to parabolic systems on the whole space $\bbR^d$ for $\ell^\infty$-dissipative potentials. 
\end{enumerate}

\subsection*{Difficulties in the analysis of the long-term behaviour}

Heuristically, the long-term behaviour of~\eqref{eq:the-equation} is more involved than in the scalar-valued case since, even in cases where the potential $V$ has only real entries and is assumed to preserve boundedness of all solutions, the coupling by $V$ can cause the existence of periodic solutions.
This is for instance the case in the simple example where $N = 2$, $\calA_1 \coloneqq \calA_2 \coloneqq \Delta$ and 
\begin{align*}
	V(x) 
	\coloneqq
	\begin{pmatrix}
		0 & -1 \\ 
		1 &  0
	\end{pmatrix} 
	\qquad 
	\text{for all } x \in \Omega.
\end{align*} 
In this example, dissipativity of $V$ with respect to the Euclidean norm on $\bbR^2$ implies that all solutions to~\eqref{eq:the-equation} are bounded (see Section~\ref{section:setting-the-stage} for details). However,
\begin{align*}
u(t) = 
	\begin{pmatrix}
		u_1(t) \\ 
		u_2(t)
	\end{pmatrix}
	\coloneqq
	\begin{pmatrix}
		\cos(t) \one \\
		\sin(t) \one
	\end{pmatrix}
	\qquad 
	\text{for all } t \ge 0.
\end{align*}
is a periodic solution to~\eqref{eq:the-equation} (where $\one$ denotes the constant function on $\Omega$ with value $1$).

This kind of behaviour is not difficult to understand for potentials $V$ which do not depend on the spatial variable and for elliptic operators $\calA_1, \ldots, \calA_N$ which all coincide (see Subsection~\ref{subsection:constant-potential} for details). However, the matter becomes more delicate if one of these two properties is not satisfied.

\subsection*{Contributions in this article}

The purpose of this paper is to give a variety of conditions for the convergence to equilibrium of the solutions to~\eqref{eq:the-equation}. 
After some preliminaries in Section~\ref{section:setting-the-stage}, several convergence criteria are established in Section~\ref{section:convergence}. The third of these criteria is probably the most surprising one, and it is particularly interesting for the cases $p=1$ and $p=\infty$ in which the condition is easily checkable.

\begin{enumerate}[\hspace{6pt} (a)]
	\item If all solutions to~\eqref{eq:the-equation} are bounded (say, in $L^p(\Omega;\bbC^N)$ for some $p \in [1,\infty]$), each matrix $V(x)$ is real, and the off-diagonal entries of each matrix $V(x)$ are positive, then every solution to~\eqref{eq:the-equation} converges as $t \to \infty$ (Subsection~\ref{subsection:quasi-positive-potentials}).
	
	\item If each matrix $V(x)$ is $\ell^2$-dissipative on $\bbC^N$, convergence of the solutions can be characterized purely in terms of the spectral structure of the matrices $V(x)$ (Subsection~\ref{subsection:convergence-for-2-dissipative-potentials}).
	
	\item If each matrix $V(x)$ is real and $\ell^p$-dissipative on $\bbR^N$ for a fixed number $p \in [1,\infty] \setminus \{2\}$, all solutions converge as $t \to \infty$ (Subsection~\ref{subsection:convergence-for-p-dissipative-potentials}).
	
	\item If all elliptic operators $\calA_k$ coincide and the matrices $V(x)$ do not depend on $x$, the vector-valued differential operator commutes with the coupling potential. The long-term behaviour of the solutions to~\eqref{eq:the-equation} can thus be characterized in terms of the long-term behaviour of the semigroup on $\bbC^N$ generated by $V$. 
	
	\item
	If all elliptic operators $\calA_k$ coincide and the matrices $V(x)$ are simultaneously diagonalizable, then the equation can be decomposed into scalar-valued subsystems, for which sufficient conditions for convergence are easier to find.
\end{enumerate}

The main results are criteria~(a)--(c) which will be presented in the Subsections~\ref{subsection:quasi-positive-potentials}--\ref{subsection:convergence-for-p-dissipative-potentials}, respectively.
The remaining criteria~(d)~and~(e) presented in the Subsections~\ref{subsection:constant-potential} and~\ref{subsection:decoupling} are fairly specific. 
They are mainly included to cover particularly simple cases where the long-term behaviour can be analysed by easier and more direct methods than in the cases~(a)--(c).

Let us give a brief overview of the different tools which we will use to prove each of the five criteria~(a)--(e):

\begin{enumerate}[(a)]
	\item 
	This criterion will be derived from Perron--Frobenius theory for positive $C_0$-semigroups.
	
	\item 
	This follows from elementary dissipativity estimates.

	\item 
	This criterion will be derived from results in \cite{Dobrick2021, Glueck2016b} which relate the spectrum of contractive semigroups to the geometry of $L^p$-spaces.
	
	\item 
	This is mainly an exercise in linear algebra.
	
	\item 
	This criterion follows from combining a diagonalization argument with a special case of~(b).
\end{enumerate}

We stress that no new theoretical methods are developed in this article. The main point of the present article is rather to combine various abstract theorems from the literature to understand the long-term behaviour of~\eqref{eq:the-equation} in a large variety of cases.

\subsection*{Further related literature}

As explained before, equations of the type~\eqref{eq:the-equation} are closely related to reaction-diffusion equations. 
In addition, there is a branch of the literature where heat equations coupled by matrix-valued potentials are studied with the following focus:
for instance in \cite{Kunze2019, Kunze2020, Maichine2019} 
equations on the whole space $\bbR^d$ are considered where the potential $V$ is not assumed to be bounded. 
In this case, a major point of interest is well-posedness of the equation, i.e., whether the differential operator generates a $C_0$-semigroup.
In contrast, our setting is simpler in the sense that generation properties follow from elementary perturbation theory;
instead, our main interest is to give non-trivial conditions on the potential that imply convergence to equilibrium.

\subsection*{Organisation of the article} 

In Section~\ref{section:setting-the-stage} we discuss the detailed setting in which the parabolic problem~\eqref{eq:the-equation} is posed. Further, it is explained how the solution semigroup acts on the $L^p$-scale. The results from this section are needed as a proper set-up for our convergence results in Section~\ref{section:convergence}. 
In the appendix, we recall characterizations of $p$-dissipativity for finite-dimensional matrices.

\section{Setting the stage} \label{section:setting-the-stage}

\subsection{The equation} \label{subsection:the-equation}

Let $\emptyset \not= \Omega \subseteq \bbR^d$ be a bounded domain which has the \emph{extension property} in the sense that every Sobolev function in $H^1(\Omega;\bbC)$ is the restriction of a Sobolev function in $H^1(\bbR^d;\bbC)$. This is the case, e.g., if $\Omega$ has Lipschitz boundary \cite[Section~7.3.6]{Arendt2004}.

We fix an integer $N \geq 1$ (which will denote the number of coupled equations on $\Omega$) as well as measurable and bounded functions $A_1, \ldots, A_N \colon \Omega \to \bbR^{d \times d}$ and $V \colon \Omega \to \bbC^{N \times N}$. Moreover, we assume that there exists a constant $\nu > 0$ such that for all $k \in \{1,\ldots,N\}$ and almost all $x \in \Omega$, the uniform coercivity condition
\begin{align*}
	\re(\xi^T A_k(x) \overline \xi) \ge \nu \norm{\xi}_2^2
\end{align*}
holds for all $\xi \in \bbC^d$. We will study the long-term behaviour of the solutions to the coupled parabolic equation that is formally given by
\begin{align}
	\label{eq:cp-bounded}
	\frac{d}{dt}
	\begin{pmatrix}
		u_1 \\ \vdots \\ u_N
	\end{pmatrix}
	=
	\begin{pmatrix}
		 \divergence (A_1 \nabla u_1) \\ \vdots \\ \divergence (A_N \nabla u_N)
	\end{pmatrix}
	+ 
	V
	\begin{pmatrix}
		u_1 \\ \vdots \\ u_N
	\end{pmatrix},
\end{align}
subject to Neumann boundary conditions. Due to the weak regularity assumptions on the coefficients and on the boundary of $\Omega$, we use form methods to give precise meaning to the elliptic operators $u \mapsto \divergence (A_k\nabla u)$: For each $k \in \{1,\ldots,N\}$ we define a bilinear form
\begin{align*}
	a_k: & \, H^1(\Omega;\bbC) \times H^1(\Omega;\bbC) \to \bbC, \quad a_k(u,v) = \int_\Omega \nabla u^T A_k \nabla \overline{v} \dx x.
\end{align*}
This form induces a linear operator $- \calA_k \colon L^2(\Omega;\bbC) \supseteq D(\calA_k) \to L^2(\Omega;\bbC)$, and $\calA_k$ is interpreted as a realization of the differential operator $u \mapsto \divergence (A_k\nabla u)$ with Neumann boundary conditions. Moreover, each operator $\calA_k$ generates a positive (in the sense of Banach lattices) and contractive $C_0$-semigroup $(\ue^{t \calA_k})_{t \geq 0}$ on $L^2(\Omega;\bbC)$. For a general overview of form methods in the context of heat equations, we refer the reader to \cite{Ouhabaz2005}.

\subsection{Behaviour on the $L^p$-scale} \label{subsection:behaviour-on-the-lp-scale}

In this subsection, we briefly discuss how those semigroups act on the $L^p$-scale. We will see that, due to an ultracontractivity argument, most of the relevant properties do not depend on the choice of $p$. The arguments in this subsection are fairly standard, but there are a few subtleties since we also want to consider the respective semigroups for $p = \infty$. Thus, all the relevant properties of the semigroups are stated in detail.

We start with the observation that the constant function $\one$ is a fixed point of each semigroup $(\ue^{t \calA_k})_{t \geq 0}$ and its dual. Thus, it follows from interpolation theory that these semigroups induce positive and contractive $C_0$-semigroups on the $L^p$-scale, where $p \in [1,\infty)$. 
We denote the generators corresponding to those semigroups by $\calA_{k,p}$. In particular, one has $\calA_{k,2} = \calA_k$.

The coupled parabolic equation~\eqref{eq:cp-bounded} can now precisely be stated as the abstract Cauchy problem
\begin{align}
	\label{eq:coupled-heat-equation-on-bounded-domain-general}
	\frac{d}{dt} u = \calB_p u + Vu
\end{align}
on the Banach space $L^p(\Omega; \bbC^N)$, where $p \in [1,\infty)$ and 
\begin{align*}
	\calB_p =
	\begin{pmatrix}
		\calA_{1,p} &  & \\
		 &  \ddots &  \\
		 &  & \calA_{N,p}
	\end{pmatrix}.
\end{align*}
In the following, we endow $L^p(\Omega;\bbC^N)$, $p \in [1, \infty]$, with the norm $\norm{\argument}_p$ given by 
\begin{equation}
	\label{eq:vector-valued-lp-norm}
	\begin{aligned}
		\norm{u}_p^p &= \int_\Omega \norm{u(x)}_p^p \dx x = \sum_{k=1}^N \norm{u_k}_{L^p(\Omega)}^p \qquad \text{for } p \in [1,\infty), \text{ and} \\
		\norm{u}_\infty &= \max\{\norm{u_k}_\infty: \; k \in \{1,\ldots,N\}\}
	\end{aligned}
\end{equation}
for $u = (u_1,\ldots,u_N) \in L^p(\Omega;\bbC^N)$. This has the following simple but important consequence.

\begin{remark} \label{rem:advantage-of-choice-of-norm}
	Let $p \in [1,\infty]$. The norm defined in~\eqref{eq:vector-valued-lp-norm} is of course equivalent to the norm that we would obtain by endowing $\bbC^N$ with the Euclidean norm and then endowing $L^p(\Omega;\bbC^N)$ with the vector-valued $p$-norm. 
	
	However, the main advantage of the norm $\norm{\argument}_p$ defined in~\eqref{eq:vector-valued-lp-norm} is that it renders $L^p(\Omega;\bbC^N)$ isometrically lattice isomorphic to the $L^p$-space of scalar-valued functions over $N$ disjoint copies of $\Omega$, i.e., we can treat $L^p(\Omega;\bbC^N)$ as a scalar-valued $L^p$-space. In particular, the geometry of $L^p(\Omega;\bbC^N)$ coincides with that of a scalar-valued $L^p$-space. 
\end{remark}

In what follows, we will use the symbol $V$ both to denote the function $V \colon \Omega \to \bbC^{N \times N}$ that was introduced in the previous subsection and the operator $L^p(\Omega;\bbC^{N}) \to L^p(\Omega;\bbC^N)$ given by multiplication with this function (for any $p \in [1,\infty]$).

Since $V$ is a bounded operator for each $p$, it follows from standard perturbation theory that $\calB_p+V$ generates a $C_0$-semigroup $(\ue^{t(\calB_p + V)})_{t \geq 0}$ on $L^p(\Omega; \bbC^N)$ for each $p \in [1,\infty)$. 

The semigroups $(\ue^{t(\calB_p + V)})_{t \geq 0}$ are consistent on the $L^p$-scale. This follows from a perturbation argument (e.g., by utilizing Trotter's product formula or the Dyson--Phillips series) since the semigroups generated by $\calB_p$ are consistent. 

Moreover, $L^\infty(\Omega;\bbC^N)$ is invariant under the action of the semigroups $(\ue^{t(\calB_p + V)})_{t \geq 0}$ as the following proposition shows. For a proper reading of the proposition, note that the realizations of the multiplication operator $V$ as bounded operators on $L^p(\Omega;\bbC^N)$ are consistent for $p \in [1,\infty]$. Moreover, the exponential operators $\ue^{tV}$ are, for every $t \geq 0$, also consistent on the $L^p(\Omega;\bbC^N)$-scale; in other words, for $1 \le p \le q \le \infty$, it does not make a difference whether we consider the exponential $\ue^{tV}$ on $L^p(\Omega;\bbC^N)$ first and then restrict it to $L^q(\Omega;\bbC^N)$ or whether we consider it on $L^q(\Omega;\bbC^N)$ in the first place.

\begin{proposition} \label{prop:l_infty-norm-of-semigroup-on-bounded-domain}
	There exists a number $\omega \in \mathbb{R}$ such that $\norm{\ue^{tV}}_{\infty \to \infty} \le \ue^{t\omega}$ for all $t \geq 0$. For any such $\omega$, for each $t \geq 0$ and for each $p \in [1,\infty)$, the operator $\ue^{t(\calB_p+V)}$ on $L^p(\Omega;\bbC^N)$ leaves $L^\infty(\Omega;\bbC^N)$ invariant and satisfies $\norm{\ue^{t(\calB_p+V)}}_{\infty \to \infty} \le \ue^{t\omega}$.
\end{proposition}
\begin{proof}
	The existence of $\omega$ follows from the fact that $\norm{\ue^{tV}}_{\infty \to \infty} \le \ue^{t\norm{V}_{\infty \to \infty}}$ for all $t \geq 0$.
	
	Now, fix such an $\omega$ as well as $t \geq 0$ and $p \in [1,\infty)$. By Trotter's product formula (cf.\ \cite[Corollary~III.V.8]{Engel2000}) we have $\ue^{t(\calB_p+V)}f = \lim_{n \to \infty} (\ue^{\frac{t}{n}\calB_p}\ue^{\frac{t}{n}V})^nf$ with respect to the $L^p$-norm for each $f \in L^p$. The semigroup generated by $\calB_p$ is $L^\infty$-contractive. Thus, if $f$ is an element of the closed unit ball $\mathrm B[0, 1]$ of $L^\infty$, then $(\ue^{\frac{t}{n}\calB_p}\ue^{\frac{t}{n}V})^nf$ is an element of $\ue^{t\omega} \mathrm B[0, 1]$ for each $n \in \bbN$ and so is the limit as $n \to \infty$ since $\mathrm B[0, 1]$ is closed in $L^p$.
\end{proof}

Since the semigroups act consistently on the $L^p$-scale, the restriction of the operator $\ue^{t(\calB_p+V)}$ to $L^\infty(\Omega;\bbC^N)$ is the same operator for all $p \in [1,\infty)$. From now on, by abuse of notation, the restriction of $\ue^{t(\calB_p+V)}$ to $L^\infty(\Omega;\bbC^N)$ is denoted by $\ue^{t(\calB_\infty + V)}$. Note that this is used purely as a notation. In particular, no operator $\calB_\infty$ is defined and no assertions about such an operator are made. Clearly, $(\ue^{t(\calB_\infty + V)})_{t \geq 0}$ is an operator semigroup, but it is certainly not strongly continuous in general. However, it follows from Proposition~\ref{prop:ultracontractive} below that this semigroup is strongly continuous and, in fact, even norm continuous on the open time interval $(0,\infty)$.

\begin{proposition} \label{prop:ultracontractive}
	Let $p \in [1,\infty]$. Then the following assertions hold:
	\begin{enumerate}[\upshape (i)]
		\item For each $t \in (0,\infty)$ the operator $\ue^{t(\calB_p + V)}$ maps $L^p(\Omega; \bbC^N)$ boundedly into $L^\infty(\Omega; \bbC^N)$, i.e., it is a bounded operator from $L^p(\Omega;\bbC^N)$ to $L^\infty(\Omega;\bbC^N)$.
		
		\item For each $t \in (0,\infty)$ the operator $\ue^{t(\calB_p + V)}$ is compact on $L^p(\Omega; \bbC^N)$.
	\end{enumerate}
\end{proposition}
\begin{proof}
	(i) We consider the generator $\calB_2+V$ on $L^2(\Omega;\bbC^N)$. The operator $-\calB_2$ is associated with the form 
	\begin{align*}
	b \colon H^1(\Omega; \bbC^N) \times H^1(\Omega; \bbC^N) \to \bbC, \quad b(u,v) = \sum_{k=1}^N a_k(u_k, v_k).
	\end{align*}
	Hence, $-(\calB_2 + V)$ is associated with the form 
	\begin{align*}
	c \colon H^1(\Omega; \bbC^N) \times H^1(\Omega; \bbC^N) \to \bbC, \quad c(u,v) = b(u, v) - \int_\Omega \langle Vu, v \rangle \dx x.
	\end{align*}	
	As $\Omega$ has the extension property, $H^1(\Omega;\bbC)$ embeds continuously into $L^q(\Omega;\bbC)$ for some $q > 2$, and hence, the form domain $H^1(\Omega;\bbC^N)$ embeds continuously into $L^q(\Omega;\bbC^N)$. Thus, by ultracontractivity (c.f.\ the theorem in \cite[Section~7.3.2]{Arendt2004}) it follows that $\ue^{t(\calB_p + V)}$ maps $L^p(\Omega;\bbC^N)$ into $L^\infty(\Omega;\bbC^N)$ for each $p \in [1,\infty)$. The boundedness of this mapping follows also by ultracontractivity or, alternatively, from the closed graph theorem.
	
	(ii) For $p \in [1,\infty)$ this is a consequence of~(i) and Dunford--Pettis theory, see for instance~\cite[Theorem~7.1]{GlueckSG} for details. For $p = \infty$, observe that $\ue^{t(\calB_\infty + V)}$ factors as
	\begin{align*}
		L^\infty(\Omega;\bbC^N) \xrightarrow{\id} L^2(\Omega;\bbC^N) \xrightarrow{\ue^{\frac{t}{2}(\calB_2 + V)}} L^2(\Omega;\bbC^N) \xrightarrow{\ue^{\frac{t}{2}(\calB_2 + V)}} L^\infty(\Omega;\bbC^N).
	\end{align*}
	Hence, the claim follows from the case $p = 2$.
\end{proof}

As two consequences of the above proposition, boundedness and operator norm convergence of the semigroup does not depend on the choice of $p$.

\begin{corollary} \label{cor:ultracontractive-bounded}
	Set $L^p \coloneqq L^p(\Omega; \bbC^N)$ for every $p \in [1, \infty]$. The following assertions are equivalent:
	\begin{enumerate}[\upshape (i)]
		\item There exists $p \in [1,\infty]$ such that the semigroup $(\ue^{t(\calB_p+V)})_{t \geq 0}$ is bounded on $L^p$.
		
		\item For every $p \in [1,\infty]$ the semigroup $(\ue^{t(\calB_p+V)})_{t \geq 0}$ is bounded on $L^p$.
	\end{enumerate}
\end{corollary}
\begin{proof}
	Obviously,~(ii) implies~(i), so assume conversely that~(i) holds and consider any $q \in [1,\infty]$. For each $t \ge 2$, the operator $\ue^{t(\calB_q+V)}$ factors as
	\begin{align*}
		L^q \xrightarrow{\ue^{\calB_q+V}} L^\infty \xrightarrow{\id} L^p \xrightarrow{\ue^{(t-2)(\calB_p+V)}} L^p \xrightarrow{\ue^{\calB_p+V}} L^\infty \xrightarrow{\id} L^q.
	\end{align*}
	Therefore, $\sup_{t \in [2,\infty)} \norm{\ue^{t(\calB_q+V)}} < \infty$. On the other hand, observe that
	\begin{align*}
		\sup_{t \in [0,2]} \big \lVert \ue^{t(\calB_q+V)} \big \rVert < \infty;
	\end{align*}
	this follows from the $C_0$-property for $q \in [1,\infty)$ and from Proposition~\ref{prop:l_infty-norm-of-semigroup-on-bounded-domain} for $q = \infty$.
\end{proof}

\begin{corollary} \label{cor:ultracontractive-convergent}
	Set $L^p \coloneqq L^p(\Omega; \bbC^N)$ for every $p \in [1, \infty]$. The following assertions are equivalent:
	\begin{enumerate}[\upshape (i)]
		\item There exists $p \in [1,\infty]$ such that $\ue^{t(\calB_p + V)}$ converges with respect to the operator norm on $L^p$ as $t \to \infty$.
		
		\item For every $p \in [1,\infty]$ the operator $\ue^{t(\calB_p + V)}$ converges with respect to the operator norm on $L^p$ as $t \to \infty$.
		
		\item For every $p \in [1,\infty]$ the operator $\ue^{t(\calB_p + V)}$ converges with respect to the operator norm in $\calL(L^p; L^\infty)$ as $t \to \infty$.
	\end{enumerate}
\end{corollary}

\begin{proof}
	Obviously,~(iii) implies~(ii) and~(ii) implies~(i). 
	
	(i) $\Rightarrow$ (iii): Let $q \in [1, \infty]$. For $t \geq 2$ the operator $\ue^{t(\calB_q+V)} \coloneqq L^q \to L^\infty$ factors as
	\begin{align*}
		L^q \xrightarrow{\ue^{\calB_q+V}} L^\infty \xrightarrow{\id} L^p \xrightarrow{\ue^{(t-2)(\calB_p+V)}} L^p \xrightarrow{\ue^{\calB_p+V}} L^\infty.
	\end{align*}
	Therefore, $\ue^{t(\calB_p+V)}$ converges in $\calL(L^q; L^\infty)$ with respect to the operator norm as $t \to \infty$.
\end{proof}

Corollary~\ref{cor:ultracontractive-bounded} and~\ref{cor:ultracontractive-convergent} show that if the solutions to the coupled Cauchy problem~\eqref{eq:coupled-heat-equation-on-bounded-domain-general} are bounded or converge uniformly in one $p$-norm, they are bounded or converge uniformly in every $p$-norm, respectively. This motivates the following terminology that will be used throughout the rest of the article.

\begin{definition}
	\label{def:uniform-boundedness-convergence-for-coupled-heat-equation}
	\begin{enumerate}[\hspace{6pt} \upshape (a)] 
		\item
		We say that the solutions to the coupled heat 
		equation~\eqref{eq:coupled-heat-equation-on-bounded-domain-general} 
		are \emph{uniformly bounded} if one, and thus all, 
		of the equivalent assertions from 
		Corollary~\ref{cor:ultracontractive-bounded} are satisfied.
		
		\item
		We say that the solutions to the coupled heat equation~\eqref{eq:coupled-heat-equation-on-bounded-domain-general} 
		\emph{converge uniformly} as $t \to \infty$ if one, 
		and thus all, of the equivalent assertions of 
		Corollary~\ref{cor:ultracontractive-convergent} are satisfied.
	\end{enumerate}
\end{definition}

In Section~\ref{section:convergence} we provide sufficient conditions for the uniform convergence in the sense of Definition~\ref{def:uniform-boundedness-convergence-for-coupled-heat-equation} of the solutions to~\eqref{eq:coupled-heat-equation-on-bounded-domain-general} as $t \to \infty$.

\subsection{Dissipativity} \label{subsection:dissipativity}

In view of Corollary~\ref{cor:ultracontractive-bounded}, boundedness of the solution semigroup to~\eqref{eq:coupled-heat-equation-on-bounded-domain-general} for one $p$ implies boundedness on the entire $L^p$-scale. The easiest way to obtain boundedness for some $p \in [1,\infty]$ is to assume that the multiplication operator $V$ is dissipative on $L^p(\Omega;\bbC^N)$. This is discussed in a detail in this subsection. For a general treatment of the concept \emph{dissipativity}, including its definition and various characterizations, we refer for instance to \cite[Section~II.3.b]{Engel2000}.

\begin{proposition} \label{prop:dissipative-potential-implies-contractive-semigroup}
	Let $p \in [1,\infty]$. If $V$ is dissipative on $L^p(\Omega;\bbC^N)$, then the semigroup $(\ue^{t(\calB_p + V)})_{t \geq 0}$ is contractive on $L^p(\Omega;\bbC^N)$.
\end{proposition}
\begin{proof}
	First, assume that $p \in [1,\infty)$. Then $\calB_p$ generates a contractive $C_0$-semigroup on $L^p(\Omega;\bbC^N)$. It follows from the characterization of dissipativity in \cite[Proposition~II.3.23]{Engel2000} that $\calB_p+V$ is dissipative, too. 
	
	If $p = \infty$, then $\norm{\ue^{tV}}_{\infty \to \infty} \le 1$ for each $t \geq 0$. Hence, the assertion follows from Proposition~\ref{prop:l_infty-norm-of-semigroup-on-bounded-domain}.
\end{proof}

Proposition~\ref{prop:dissipative-potential-implies-contractive-semigroup} shows why dissipativity of the multiplication operator $V$ on $L^p(\Omega;\bbC^N)$ is of interest in our setting. 
Next, we note that dissipativity of $V$ can be characterized in terms of the matrices $V(x)$:

\begin{proposition} \label{prop:p-dissipative}
	For each $p \in [1, \infty]$ the following assertions are equivalent:
	\begin{enumerate}[\upshape (i)]
		\item The multiplication operator $V$ on $L^p(\Omega;\bbC^N)$ is dissipative.
		\item For almost all $x \in \Omega$, the matrix $V(x)$ is dissipative with respect to the $\ell^p$-norm on $\bbC^N$.
	\end{enumerate}
	If, in addition, $V(x) \in \bbR^{N \times N}$ almost all $x \in \Omega$, then the above assertions are also equivalent to:
	\begin{enumerate}[\upshape (i)] \setcounter{enumi}{2}
		\item For almost all $x \in \Omega$, the matrix $V(x)$ is dissipative with respect to the $\ell^p$-norm on $\bbR^N$.
	\end{enumerate}
\end{proposition}
\begin{proof}
	The equivalence of~(i) and~(ii) is an immediate consequence of our choice of the norm on $L^p(\Omega;\bbC^d)$ (see formula~\eqref{eq:vector-valued-lp-norm}).
	
	Now assume that the matrix $V(x)$ has only real entries for almost all $x \in \Omega$.
	The implication from~(ii) to~(iii) is obvious.
	To show that~(iii) implies~(ii) note that, for every matrix $M \in \bbR^{N \times N}$, its operator norm induced by the $p$-norm on $\bbR^N$ coincides with its operator norm induced by the $p$-norm on $\bbC^N$. Indeed, for $p \in [1,\infty)$ this is \cite[Proposition~2.1.1]{Fendler1998}, and for $p = \infty$ this follows from the identity
	\begin{align*}
		\norm{\xi}_\infty = \sup_{\theta \in [0,2\pi]} \norm{\re(\ue^{\ui\theta}\xi)}_\infty 
		\qquad \text{for all } \xi \in \bbC^N. 
	\end{align*}
	So overall, if $\ue^{tV(x)}$ is contractive on $(\bbR^N, \norm{\argument}_p)$, then it is also contractive on $(\bbC^N, \norm{\argument}_p)$, i.e., (iii) implies~(ii).
\end{proof}

In view of Proposition~\ref{prop:p-dissipative} it is worthwhile to recall that dissipativity of matrices with respect to the $\ell^p$-norm on $\bbR^N$ can be characterized quite explicitly. For the convenience of the reader, said characterization is presented in Proposition~\ref{prop:characterization-of-ell_p-dissipative-matrices} in the appendix.

\section{Convergence to equilibrium} \label{section:convergence}

In this section we prove several criteria under which the solutions to the coupled heat equation~\eqref{eq:coupled-heat-equation-on-bounded-domain-general} converge uniformly in the sense of Definition~\ref{def:uniform-boundedness-convergence-for-coupled-heat-equation}(b). In order to provide a simple but illuminating illustration and comparison of our results, in each subsection the long-term behaviour of the solutions to a concrete, simple toy example is discussed. Namely, we consider the evolution equation
\begin{align}
	\label{eq:toy-example}
	\begin{pmatrix}
		\dot u_1 \\ \dot u_2
	\end{pmatrix}
	=
	\begin{pmatrix}
		\Delta u_1 \\ \Delta u_2
	\end{pmatrix}
	+
	V
	\begin{pmatrix}
		u_1 \\ u_2
	\end{pmatrix}
\end{align}
for a potential $V \colon \Omega \to \bbR^{2 \times 2}$ subject to Neumann boundary conditions on $\Omega$. However, the results from these sections hold, of course, for much more general settings.

The following table provides an overview of the following subsections and the respective convergence results proven in each of them.

\begin{figure}[h]
	\begin{center}
		\begin{tabularx}{\textwidth}{ |c|>{\centering}X|>{\centering}X|c| } 
			\hline
			
			Section &
			Further conditions on the coefficients $A_k$ &
			Further conditions on the potential $V$ &
			Convergence results

			\\
			\hline
			
			3.1 &
			none &
			$V$ real and quasi-positive a.e.; solutions bounded &
			\begin{tabular}{c}
				Theorem~\ref{theorem:convergence-quasipositive-potential}, \\
				Proposition~\ref{prop:boundedness-quasipositive-potential}
			\end{tabular}
			
			\\ 
			\hline
			
			3.2 &
			none &
			$V$ $\ell^2$-dissipative a.e. &
			Theorem~\ref{thm:convergence-for-coupled-heat-equation-l2-case} 
			
			\\ 
			\hline 
			
			3.3 &
			none &
			$V$ real and $\ell^p$-dissipative a.e.\ for some $p \neq 2$ &
			Theorem~\ref{thm:convergence-for-coupled-heat-equation-lp-case} 
			
			\\ 
			\hline
			
			3.4 &
			all $A_k$ coincide &
			$V$ constant a.e. &
			Proposition~\ref{prop:constant-potential} 
			
			\\ 
			\hline 
			
			3.5 &
			all $A_k$ coincide &
			$V$ simultaneously diagonalizable a.e., real part of all eigenvalues is $\le 0$ a.e.\ &
			Proposition~\ref{prop:decoupled-system} 
			
			\\ 
			\hline
		\end{tabularx}
	\end{center}
\end{figure}

First, we turn our attention to situations in which the long-term behaviour of the system can be investigated by employing arguments from classical Perron--Frobenius theory. 

\subsection{Convergence for quasi-positive potentials} 
\label{subsection:quasi-positive-potentials}

A matrix $A \in \bbR^{N \times N}$ is called \textit{quasi-positive} if all off-diagonal entries of $A$ are $\ge 0$. In this section, it is shown that the solutions to the coupled parabolic equation~\eqref{eq:coupled-heat-equation-on-bounded-domain-general} converge uniformly if they are bounded and if the potential $V$ is quasi-positive almost everywhere. The key insight is that the quasi-positivity of $V$ yields that the solution semigroup is positive.
We call an element $f \in L^2(\Omega;\bbC^N)$ \emph{positive} if, for each $k \in \{1, \ldots, N\}$, the component function $f_k$ 
takes values in $[0,\infty)$ almost everywhere. 
A linear operator on $L^2(\Omega;\bbC^N)$ is called \emph{positive} if it maps each positive function to a positive function.

\begin{proposition} 
	\label{prop:quasipositive-potential}
	If $V(x)$ is in $\bbR^{N \times N}$ and quasi-positive for almost every $x \in \Omega$, then the semigroup $(\ue^{t(\calB_2 + V)})_{t \geq 0}$ consists of positive operators. 
\end{proposition}

\begin{proof}
	As $V \colon \Omega \to \bbR^{N \times N}$ is bounded and quasi-positive almost everywhere, there exists $c > 0$ such that $V + c$ is positive almost everywhere; 
	here, we use $V+c$ as an abbreviation for $V + c\id$ where $\id$ is the $N \times N$-identity matrix. 
	Hence, by standard perturbation theory, the semigroup $(\ue^{t(\calB_2 + V + c)})_{t \geq 0}$ is positive. By rescaling one thus obtains
	\begin{align*}
		\ue^{t(\calB_2 + V)} = \ue^{- c t} \ue^{t(\calB_2 + V + c)} \geq 0 
		\qquad \text{for all } t \geq 0,
	\end{align*}
	i.e., $(\ue^{t(\calB_2 + V)})_{t \geq 0}$ consists of positive operators.
\end{proof}

\begin{theorem} 
	\label{theorem:convergence-quasipositive-potential}
	Suppose that $V(x)$ is in $\bbR^{N \times N}$ and quasi-positive for almost every $x \in \Omega$ 
	and that the solutions to the coupled parabolic equation~\eqref{eq:coupled-heat-equation-on-bounded-domain-general} are uniformly bounded. 
	Then the solutions to~\eqref{eq:coupled-heat-equation-on-bounded-domain-general} 
	converge uniformly as $t \to \infty$. 
\end{theorem}

\begin{proof}
	According to Proposition~\ref{prop:ultracontractive}(ii), 
	the semigroup $(\ue^{t(\calB_2 + V)})_{t \geq 0}$ is immediately compact. 
	Hence, the positivity of $(\ue^{t(\calB_2 + V)})_{t \geq 0}$ implies 
	that the so-called \emph{boundary spectrum} 
	\begin{align*}
		\spec(\calB_2 + V) \cap (\mathrm s(\calB_2 + V) + \ui \bbR)
	\end{align*}
	of $\calB_2 + V$ contains only the spectral bound $\mathrm s(\calB_2 + V)$ of $\calB_2 + V$ 
	(cf.\ \cite[Corollary~C-III.2.12]{Arendt1986}). 
	
	If $\mathrm s(\calB_2 + V) < 0$, then the immediate compactness of the semigroup implies 
	that it converges uniformly to $0$ 
	(as immediately compact semigroups satisfy the spectral mapping theorem, 
	see e.g.\ \cite[Corollary~IV.3.12]{Engel2000}).
	So assume now that $\mathrm s(\calB_2 + V) = 0$. 
	Then, according to what we just showed, $0$ is the only spectral value of $\calB_2 + V$ 
	on the imaginary axis; 
	it is a pole of the resolvent (due to the immediate compactness of the semigroup) 
	and this pole is of order $1$ since the semigroup is bounded.  
	This implies the operator norm convergence of the as $t \to \infty$, 
	see e.g.\ \cite[Proposition~V.4.3]{Engel2006}.
	So, the solutions to~\eqref{eq:coupled-heat-equation-on-bounded-domain-general} converge uniformly as $t \to \infty$.
\end{proof}

A related result for coupled parabolic equations on the whole space $\bbR^d$ was proved in \cite[Theorem~4.4]{Addona2019}. There, the matrices in the potential are supposed to be quasi-positive, but the assumptions on the differential operator and its coefficients are distinct from those in the present article.

Next, a sufficient condition for uniform boundedness of the solutions to equation~\eqref{eq:coupled-heat-equation-on-bounded-domain-general} is presented. This condition guarantees that Theorem~\ref{theorem:convergence-quasipositive-potential} can be applied.

\begin{proposition} 
	\label{prop:boundedness-quasipositive-potential}
	Suppose that $V(x)$ is in $\bbR^{N \times N}$ and is quasi-positive for almost every $x \in \Omega$ 
	and that there exists a vector $z \in \bbR^N$ whose components are all strictly greater than $0$ and 
	which satisfies $V(x)z = 0$ for almost all $x \in \Omega$.
	Then the solutions to equation~\eqref{eq:coupled-heat-equation-on-bounded-domain-general} 
	are uniformly bounded and thus converge uniformly as $t \to \infty$.  
\end{proposition}

\begin{proof}
	The vector $z \, \otimes \one \coloneqq (z_1\one, \ldots, z_N\one)^T \in L^2(\Omega;\bbC^N)$ is in the kernel of $\calB_2 + V$ and thus a fixed vector of the semigroup $(\ue^{t(\calB_2 + V)})_{t \ge 0}$ on $L^2(\Omega;\bbC^N)$.
	
	Moreover, each vector in the real space $L^\infty(\Omega;\bbR^N)$ is bounded above and below by a multiple of $z \otimes \one$ and thus, by the positivity of the semigroup $(\ue^{t(\calB_2 + V)})_{t \ge 0}$, the orbit of the vector is bounded in $L^\infty(\Omega;\bbR^N)$.
	Consequently, the same is true for every vector in $L^\infty(\Omega;\bbC^N)$. Hence, the semigroup $(\ue^{t(\calB_\infty + V)})_{t \ge 0}$ on $L^\infty(\Omega;\bbC^N)$ is bounded by the uniform boundedness principle. Thus, Theorem~\ref{theorem:convergence-quasipositive-potential} yields the claim.
\end{proof}

Recall that a positive $C_0$-semigroup on an $L^p$-space is called \emph{irreducible} 
if the only invariant closed ideals are $\{0\}$ and the entire space 
(see for instance \cite[Section~C-III-3]{Arendt1986} or \cite[Section~14.3]{Batkai2017} for details).
\emph{Irreducibility} for a positive linear operator is defined analogously.

\begin{remark}
	\label{rem:irreducible}
	Let the assumptions of Theorem~\ref{theorem:convergence-quasipositive-potential} 
	or Proposition~\ref{prop:boundedness-quasipositive-potential} be satisfied.
	Define a matrix $W \in \bbR^{N \times N}$ as follows: 
	for all indices $j,k \in \{1, \ldots, N\}$ we set $W_{jk} \coloneqq 0$ if the component $V_{jk} \in L^\infty(\Omega;\bbR)$ of $V$
	is $0$ almost everywhere and $W_{jk} \coloneqq 1$ otherwise. 
	
	If the matrix $W$ is irreducible, then the limit projection $P$ on $L^p(\Omega;\bbC^N)$ of the semigroup $(\ue^{t(\calB_p + V)})_{t \geq 0}$ is either $0$ or has rank $1$.
\end{remark}

\begin{proof}
	Let us first consider the case $p \in [1,\infty)$ (since we have a $C_0$-semigroup in this case).
	Since $\Omega$ is connected, the heat semigroup generated by $\calA_k$ on $L^p(\Omega;\bbC)$ is irreducible. 
	As the matrix $W$ is also irreducible, we can hence apply the perturbation result in \cite[Proposition~C-III-3.3]{Arendt1986} 
	to see that the semigroup generated by $\calB_2 + V$ on $L^p(\Omega;\bbC^N)$ is also irreducible; 
	this argument is taken from \cite[Proposition~in~Section~8]{Nagel1989}.
	The irreducibility of the semigroup implies that the limit operator 
	is either $0$ or has rank $1$; 
	this follows from classical arguments in Perron--Frobenius theory, see for instance \cite[Proposition~3.1(c)]{Arendt2020} 
	for a detailed explanation.
	
	Since the semigroups generated by $\calB_p + V$ act consistently on the $L^p$-scale, 
	the same is true for the limit operator. 
	Hence, the conclusion carries over to the case $p=\infty$.
\end{proof}

Due to the irreducibility in the preceding remark, even a bit more can actually be said about the limit operator $P$. 
We refrain from discussing this in detail here and refer to the general result explained in \cite[Proposition~3.1(c)]{Arendt2020} instead.
We conclude this subsection with a simple example.

\begin{example} \label{ex:toy-quasi-positive}
	Let $N = 2$, 
	let $a, b \colon \Omega \to [0, \infty)$ measurable and bounded, 
	and let the potential $V$ be given by
	\begin{align*}
		V(x) = 
		a(x)
		\begin{pmatrix}
			-1 &  \phantom{-}2 \\
			\phantom{-}2 & -4
		\end{pmatrix}
		+
		b(x)
		\begin{pmatrix}
			-1 & \phantom{-}2 \\
			\phantom{-}1 & -2
		\end{pmatrix} 
		\qquad \text{for all } x \in \Omega.
	\end{align*}
	Each matrix $V(x)$ is quasi-positive and the vector $z \coloneqq (2,1)^{\operatorname{T}}$ is in the kernel of each $V(x)$.
	Thus, it follows from Proposition~\ref{prop:boundedness-quasipositive-potential} that the solutions to the evolution equation~\eqref{eq:toy-example} converge uniformly as $t \to \infty$. Since the function $u_0 \coloneqq \one \cdot z$ is an equilibrium, the limit is non-zero for some initial values. 
	Remark~\ref{rem:irreducible} shows that if at least one of the functions $a$ and $b$ is non-zero on a set of strictly positive measure, then the limit operator has rank $1$.
	
	One interesting aspect of the potential matrices $V(x)$ in this example is that they are, in general, not simultaneously diagonalizable (and hence, the results from Subsection~\ref{subsection:decoupling} below cannot be applied); this follows from the fact that the respective second eigenspaces (the ones not spanned by $z$) of the matrices
	\begin{align*}
		\begin{pmatrix}
			-1 & \phantom{-}2 \\
			\phantom{-}2 & -4
		\end{pmatrix}
		\quad \text{and} \quad
		\begin{pmatrix}
			-1 & \phantom{-} 2 \\
			\phantom{-}1 & -2
		\end{pmatrix}
	\end{align*}
	are distinct.
\end{example}

\subsection{Convergence for $\ell^2$-dissipative potentials} \label{subsection:convergence-for-2-dissipative-potentials}

In this subsection, convergence of the solutions to the coupled parabolic equation~\eqref{eq:coupled-heat-equation-on-bounded-domain-general} is characterized for the case where the matrices $V(x)$ are $\ell^2$-dissipative.

\begin{proposition} 
	\label{prop:eigenvalues-of-coupled-system}
	Suppose that, for almost all $x \in \Omega$, the matrix $V(x)$ is dissipative with respect to the $\ell^2$-norm on $\bbC^N$. 
	For each $\ui\beta \in \ui\bbR$ the following two assertions are equivalent:
	\begin{enumerate}[\upshape (i)]
		\item 
		The number $\ui \beta$ is in the point spectrum $\sigma_{\pnt}(\calB_2 + V)$.
		
		\item 
		There exists a measurable subset $\widetilde{\Omega} \subseteq \Omega$ of full measure 
		(i.e., the difference $\Omega \setminus \widetilde{\Omega}$ has Lebesgue measure $0$) such that
		\begin{align*}
			\bigcap_{x \in \widetilde{\Omega}} \ker(\ui \beta - V(x)) \neq \{0\}. 
		\end{align*}
	\end{enumerate}
	In this case, each component function of every eigenvector $u \in \ker\big(\ui\beta - (\calB_2 + V)\big)$ is constant almost everywhere on $\Omega$.
\end{proposition}

\begin{proof}
	(ii) $\Rightarrow$ (i): 
	Let $0 \neq z \in \bigcap_{x \in \widetilde{\Omega}} \ker(\ui \beta - V(x)) \subseteq \bbC^N$ 
	and consider the constant function $u \colon \Omega \to \bbR^N$ given by $u(x) = z$ for almost all $x \in \Omega$. 
	Then $u \in D(\calB_2 + V)$ and $(\calB_2 + V) u = \ui \beta u$, 
	and thus $\ui \beta \in  \sigma_{\pnt}(\calB_2 + V)$.
	
	(i) $\Rightarrow$ (ii): 
	Let $u \in D(\calB_2)$ be an eigenvector of $\calB_2 + V$ to the eigenvalue $\ui \beta$.
 	Then $\langle (\calB_2 + V) u, u \rangle = \ui \beta \norm{u}_2^2$, 
 	so
 	\begin{align*}
 		0 
 		= 
 		\re \langle (\calB_2 + V) u, u \rangle 
 		= 
 		\re \langle \calB_2 u, u \rangle 
 		+ 
 		\re \langle Vu, u \rangle
 		.
 	\end{align*}
	Since both operators $\calB_2$ and $V$ are dissipative, 
	it follows that $\re \langle Vu, u \rangle = 0$ and $\re \langle \calB_2 u, u \rangle = 0$.
	The latter of these equalities together with the definition of $\calB_2$ gives
	\begin{align*}
		0 
		= 
		\sum_{k=1}^N \re \langle \calA_k u_k, u_k \rangle  
		= 
		- \sum_{k=1}^N \re a_k(u_k, u_k)
	\end{align*}
	where we used the operators $\calA_k$ and the forms $a_k$ as defined in Subsection~\ref{subsection:the-equation}.
	For every $k \in \{1, \ldots, N\}$ it follows from the coercivity estimate for the matrix valued functions 
	$A_k: \Omega \to \bbR^{d \times d}$ (see Subsection~\ref{subsection:the-equation}) that 
	\begin{align*}
		- \re a_k(u_k, u_k) 
		= 
		- \int_\Omega \nabla u_k^T A_k \nabla \overline{u_k} \dx x 
		\le  
		- \nu \int_\Omega \norm{\nabla u_k}_2^2 \dx x 
		\le 
		0
		.
	\end{align*}
	Consequently, $\re a_k(u_k, u_k) = 0$ for each $k \in \{1, \ldots, N\}$, and we thus have 
	\begin{align*}
		0 = \re a_k(u_k, u_k) \ge \nu \int_\Omega \norm{\nabla u_k}_2^2 \dx x.
	\end{align*}
	Therefore, $\nabla u_k = 0$ for each $k \in \{1, \ldots, N\}$ 
	and the connectedness of $\Omega$ implies that each of the functions $u_k$ is constant almost everywhere. 
	Hence, there is a non-zero vector $z \in \bbC^N$ such that $u(x) = z$ 
	for almost every $x \in \Omega$. 
	Thus,
	\begin{align*}
		V(x)z = V(x)u(x) = \ui \beta u(x) = \ui\beta z
	\end{align*}
	for almost every $x \in \Omega$. 
	Consequently, there exists a measurable set $\widetilde \Omega \subseteq \Omega$ of full measure 
	such that $z \in \bigcap_{x \in \widetilde{\Omega}} \ker(\ui \beta - V(x))$.
\end{proof}

The previous proposition characterizes, in terms of the matrices $V(x)$, whether $\calB_2 + V$ has a non-zero imaginary eigenvalue. This yields the following characterization of uniform convergence for the solutions to the coupled heat equation~\eqref{eq:coupled-heat-equation-on-bounded-domain-general}.

\begin{theorem} \label{thm:convergence-for-coupled-heat-equation-l2-case}
	Assume that, for almost all $x \in \Omega$, the matrix $V(x)$ is dissipative with respect to the $\ell^2$-norm on $\bbC^N$. 
	Then the following assertions are equivalent:
	\begin{enumerate}[\upshape (i)]
		\item 
		The solutions to the coupled heat equation~\eqref{eq:coupled-heat-equation-on-bounded-domain-general} 
		converge uniformly as $t \to \infty$.
		
		\item 
		For every $\ui\beta \in \ui\bbR \setminus \{0\}$ 
		and every measurable subset $\widetilde{\Omega} \subseteq \Omega$ of full measure we have
		\begin{align*}
			\bigcap_{x \in \widetilde{\Omega}} \ker(\ui \beta - V(x)) = \{0\}. 
		\end{align*}
	\end{enumerate}
\end{theorem}
\begin{proof}
	According to Proposition~\ref{prop:eigenvalues-of-coupled-system}, assertion~(ii) of the theorem is equivalent to the assertion that $\calB_2+V$ does not have any non-zero eigenvalues on the imaginary axis. So, it is left to show that this is equivalent to uniform convergence of the solutions to~\eqref{eq:coupled-heat-equation-on-bounded-domain-general}:
	
	(i) $\Rightarrow$ (ii): If (ii) does not hold, then $\calB_2 + V$ has an eigenvalue $\ui\beta \in \ui\bbR \setminus \{0\}$ with eigenvector $u$. Thus, $\ue^{t(\calB_2 + V)}u$ does not converge as $t \to \infty$, which means that (i) does not hold. 
	
	(ii) $\Rightarrow$ (i): Conversely, suppose $\calB_2 + V$ does not have any eigenvalues on the imaginary axis, except for possibly $0$.

	According to Proposition~\ref{prop:ultracontractive}(ii), the semigroup $(\ue^{t(\calB_2 + V)})_{t \geq 0}$ is immediately compact. Moreover, it is contractive as $\calB_2+V$ is dissipative. 
	Thus, the same spectral theoretic argument is in the proof of Theorem~\ref{theorem:convergence-quasipositive-potential} implies that the solutions to~\eqref{eq:coupled-heat-equation-on-bounded-domain-general} converge uniformly as $t \to \infty$.
\end{proof}

Let us state the following special case of Theorem~\ref{thm:convergence-for-coupled-heat-equation-l2-case} explicitly.

\begin{corollary} \label{cor:convergence-for-coupled-heat-equation-l2-case}
	Suppose that, for almost all $x \in \Omega$, the matrices $A_1(x),\ldots,A_N(x)$ are symmetric and the matrix $V(x)$ is dissipative with respect to the $\ell^2$-norm on $\bbC^N$. If
	\begin{align*}
		\sigma(V(x)) \cap \ui\bbR \subseteq \{0\} 
		\qquad \text{for almost every } x \in \Omega,
	\end{align*}
	then the solutions to the coupled heat equation~\eqref{eq:coupled-heat-equation-on-bounded-domain-general} converge uniformly as $t \to \infty$.
\end{corollary}

\begin{example} \label{ex:toy-2-dissipative}
	Let $N=2$, 
	let $a \colon \Omega \to \bbR \setminus \{0\}$ be bounded and measurable, 
	and let the potential $V$ be given by
	\begin{align*}
		V(x) = 
		\begin{pmatrix}
			0 & -a(x) \\
			a(x) & 0
		\end{pmatrix} 
		\qquad \text{for all } x \in \Omega.
	\end{align*}
	This potential is not quasi-positive but, since each matrix $V(x)$ is $\ell^2$-dissipative, 
	all solutions to the evolution equation~\eqref{eq:toy-example} 
	are bounded on $L^p(\Omega)$ for any $p \in [1, \infty]$ (see~Corollary~\ref{cor:ultracontractive-bounded}).
	
	The spectrum of each matrix $V(x)$ is $\{-\ui a(x), \ui a(x)\}$, 
	so it follows from Theorem~\ref{thm:convergence-for-coupled-heat-equation-l2-case} 
	that all solutions to \eqref{eq:toy-example} convergence uniformly as $t \to \infty$ if and only if $a$ is not constant almost everywhere. 
\end{example}

\subsection{Convergence for $\ell^p$-dissipative potentials} 
\label{subsection:convergence-for-p-dissipative-potentials}

A drawback of the techniques employed in the preceding section is that they rely heavily on the $2$-dissipativity of the matrices $V(x)$. 
In this section, we will instead assume that the matrices $V(x)$ 
are dissipative with respect to the $\ell^p$-norm for some fixed $p \in [1, \infty]$, $p \neq 2$, 
and have real entries only. 
As Proposition~\ref{prop:ell_p-dissipative-matrices} below shows, 
this assumption is stronger than assuming $\sigma(V(x)) \cap \ui \bbR \subseteq \{0\}$ for almost all $x \in \Omega$. 
As can be seen in the simple Example~\ref{ex:toy-p-dissipative}, 
there are cases where $p$-dissipativity of the $V(x)$ is satisfied for some $p \not= 2$ while $2$-dissipativity is not.

For our analysis, we need spectral theoretic results on a class of spaces that we call \emph{projectively non-Hilbert spaces}. 
This notion is taken from \cite[Definition~3.1]{Glueck2016b}: a real Banach space $E$ is called \emph{projectively non-Hilbert} if for no contractive rank-$2$ projection $P \in \calL(E)$ the range $PE$ is isometrically isomorph to a Hilbert space. This is, for instance, the case for each real-valued $L^p$-space, $p \in [1,\infty] \setminus \{2\}$ (cf.\ \cite[Example~3.2]{Glueck2016b} and the discussion after \cite[Example~3.5]{Glueck2016b};
in the finite-dimensional case, this was already observed in \cite[Propositions~1 and~2]{Lyubich1970}). 

The following proposition is a finite dimensional special case of \cite[Theorem~3.7]{Glueck2016b}.

\begin{proposition} 
	\label{prop:ell_p-dissipative-matrices}
	Let $p \in [1,\infty] \setminus \{2\}$. If a matrix in $\bbR^{N \times N}$ is dissipative with respect to the $\ell^p$-norm on $\bbR^N$, then its spectrum intersects the imaginary axis at most in $\{0\}$.
\end{proposition}

We point out that the main idea that underlies this proposition is much older and goes back to Lyubich \cite[Theorem~1]{Lyubich1970} (see also \cite[Section~2.4]{Belitskiui1988} and \cite[Corollary~3.9]{Lyubich1997}) who formulated a closely related result in the discrete-time case.

Note that the assertion of Proposition~\ref{prop:ell_p-dissipative-matrices} fails in the case $p = 2$, e.g.\ consider the matrix 
\begin{align*}
	V \coloneqq 
	\begin{pmatrix}
		0 & -1 \\
		1 & \phantom{-}0
	\end{pmatrix}.
\end{align*}
This matrix is dissipative with respect to the $\ell^2$-norm on $\bbR^2$, but $\sigma(V) = \{-\ui, \ui\}$. Moreover, we stress that it is essential in Proposition~\ref{prop:ell_p-dissipative-matrices} that the matrices $V(x)$ have only real entries (otherwise, the operator $\ui$ on the one-dimensional space $\bbC$ is a counterexample).

\begin{theorem} 
	\label{thm:convergence-for-coupled-heat-equation-lp-case}
	Let $p\in [1,\infty] \setminus \{2\}$ and suppose that, for almost all $x \in \Omega$, 
	the matrix $V(x)$ is in $\bbR^{N \times N}$ and is dissipative with respect to the $\ell^p$-norm on $\bbR^N$. 
	Then the solutions to the coupled heat equation~\eqref{eq:coupled-heat-equation-on-bounded-domain-general} 
	converge uniformly as $t \to \infty$.
\end{theorem}
\begin{proof}
	First, consider the case $p \in [1, \infty)$. Then the strongly continuous semigroup $(\ue^{t(\calB_p + V)})_{t \geq 0}$ is contractive on $L^p(\Omega;\bbC^N)$ by Propositions~\ref{prop:dissipative-potential-implies-contractive-semigroup} and~\ref{prop:p-dissipative}, and clearly, it leaves $L^p(\Omega;\bbR^N)$ invariant. Furthermore, as the operators $\ue^{t(\calB_p + V)}$ are compact for $t \in (0,\infty)$, it follows from \cite[Corollary~3.8]{Glueck2016b} that $\ue^{t(\calB_p + V)}$ converges with respect to the operator norm as $t \to \infty$. So, the solutions to \eqref{eq:coupled-heat-equation-on-bounded-domain-general} converge uniformly as $t \to \infty$.
	
	Now suppose that $p = \infty$. Then $(\ue^{t(\calB_\infty + V)})_{t \geq 0}$ is a contractive, though not strongly continuous, semigroup on $L^\infty(\Omega; \bbC^N)$ that leaves $L^\infty(\Omega; \bbR^N)$ invariant.
	Thus, one can apply \cite[Theorem~1.1]{Dobrick2021} (which is a convergence theorem for semigroups without time continuity assumptions) to the restriction of the semigroup to the real space $L^\infty(\Omega; \bbR^N)$ to conclude the claimed convergence.
\end{proof}

\begin{remark}
	The value of $p$ enters Theorem~\ref{thm:convergence-for-coupled-heat-equation-lp-case} only as an assumption on the matrix potential $V$. The convergence of the coupled heat semigroup takes place on the entire $L^p$-scale, as shown in Corollary~\ref{cor:ultracontractive-convergent}.

	For $p = \infty$, we do not have a $C_0$-semigroup. However, we point out that we consider the case $p = \infty$ to be quite significant (rather than just an interesting side note) since the assumption that $V(x)$ be dissipative with respect to the $\ell^p$-norm is easiest to check if $p$ is either $1$ or $\infty$, see Proposition~\ref{prop:characterization-of-ell_p-dissipative-matrices}.
\end{remark}

The following simple example shows that there are situations where Theorem~\ref{thm:convergence-for-coupled-heat-equation-lp-case} can be applied, while the other results of Section~\ref{section:convergence} cannot.

\begin{example} 
	\label{ex:toy-p-dissipative}
	Let $N=2$, 
	let $a, b: \Omega \to [0, \infty)$ measurable and bounded, 
	and let the potential $V$ be given by
	\begin{align*}
		V(x) = 
		a(x)
		\begin{pmatrix}
			-1 & -1 \\
			-2 & -2
		\end{pmatrix}
		+
		b(x)
		\begin{pmatrix}
			-1 & -1 \\
			-1 & -1
		\end{pmatrix}
		\qquad \text{for all } x \in \Omega.
	\end{align*}
	Each matrix $V(x)$ is $\ell^\infty$-dissipative (Proposition~\ref{prop:characterization-of-ell_p-dissipative-matrices}). 
	So it follows from Theorem~\ref{thm:convergence-for-coupled-heat-equation-lp-case} that the solutions to the evolution equation~\eqref{eq:toy-example} converge uniformly as $t \to \infty$. The function $(\one, -\one)^{\operatorname{T}}$ is an equilibrium, so the limit is non-zero for some initial values.
	
	Again, we note that the matrices $V(x)$ are not simultaneously diagonalizable in general, since the matrices
	\begin{align*}
		\begin{pmatrix}
			-1 & -1 \\
			-2 & -2
		\end{pmatrix}
		\quad \text{and} \quad
		\begin{pmatrix}
			-1 & -1 \\
			-1 & -1
		\end{pmatrix}
	\end{align*}
	have different sets of eigenvectors. So the system cannot be uncoupled by diagonalization, in general.
	
	Moreover, we note that the matrices $V(x)$ are not $\ell^2$-dissipative in general, since a short computation shows that the symmetric part of the matrix
	\begin{align*}
		\begin{pmatrix}
			-y & -y \\
			-z & -z
		\end{pmatrix}
	\end{align*}
	always has a strictly positive eigenvalue if $y,z \in (0,\infty)$ are two distinct numbers. 
	Hence, the Hilbert space technique from Section~\ref{subsection:convergence-for-2-dissipative-potentials} is not applicable.
\end{example}

\subsection{Constant potentials} \label{subsection:constant-potential}

In this section, we discuss a setting in which the special algebraic structure of the coupled equation allows us to make use of an ad hoc argument in order to determine the long-term behaviour. 
In particular, one can show in this case that the long-term behaviour of the system is governed solely by the spectral properties of the potential.

In this section, we study the situation where
\begin{enumerate}
	\item[\upshape (a)] $A_1(x) = \ldots = A_N(x)$ for almost all $x \in \Omega$,
	\item[\upshape (b)] $V$ is constant almost everywhere.
\end{enumerate}
 
As $V$ is constant almost everywhere, there is a unique matrix which coincides with the function $V$; and abusing the notation we denote this matrix again by $V$. This means that there are two semigroups related to $V$:
\begin{enumerate}[\upshape (1)]
	\item On one hand, $(\ue^{t V})_{t \geq 0}$ defines a semigroup on $\bbC^N$. 
	\item On the other hand, $(\ue^{t V})_{t \geq 0}$ defines a semigroup on $L^2(\Omega; \bbC^N)$. 
\end{enumerate}
It is easy to see that $\norm{\ue^{t V}}_{2, \bbC^N \to \bbC^N} = \norm{\ue^{t V}}_{L^2 \to L^2}$ for all $t \geq 0$ 
(where the former norm is the one induced by the Euclidean norm on $\bbC^N$).

The next lemma shows that for constant potentials $V$, the semigroup $(\ue^{t (\calB_2 + V)})_{t \geq 0}$ is given by a tensor product of the semigroup generated by $V$ on $\bbC^N$ and the semigroup generated by $\calA_1 = \ldots = \calA_N$ on $L^2(\Omega;\bbC)$.

\begin{lemma} \label{lemma:tensor-product-semigroups}
	If the potential $V$ is constant almost everywhere and $A_1(x) = \dots = A_N(x)$ for almost all $x \in \Omega$, then one has
	\begin{align*}
		\ue^{t(\calB_2 + V)} = \ue^{t \calB_2} \ue^{t V} = \ue^{t V}\ue^{t \calB_2} 
		\qquad \text{for all } t \geq 0.
	\end{align*}
\end{lemma}

\begin{proof}
	As $V$ is constant almost everywhere and the operators on the diagonal of $\calB_2$ all coincide, it is easy to see that the operators $\calB_2$ and $V$ commute on $L^2(\Omega; \bbC^N)$, which is in turn equivalent to the statement that the resolvents of both operator commute. So as a consequence of the Post--Widder inversion formula (cf.\ \cite[Corollary~III.5.5]{Engel2000}) this implies that the semigroups $(\ue^{t \calB_2})_{t \geq 0}$ and $(\ue^{t V})_{t \geq 0}$ commute on $L^2(\Omega; \bbC^N)$. Therefore, it follows from Trotter's product formula that $\ue^{t(\calB_2 + V)} = \ue^{t \calB_2} \ue^{t V} = \ue^{t V}\ue^{t \calB_2}$ for all $t \geq 0$.
\end{proof}

As a consequence of Lemma~\ref{lemma:tensor-product-semigroups} the long-term behaviour of the semigroup $(\ue^{t (\calB_2 + V)})_{t \geq 0}$ on $L^2(\Omega)$ depends solely on the asymptotic behaviour of the semigroup $(\ue^{t V})_{t \geq 0}$ on $\bbC^N$. 

\begin{proposition} \label{prop:constant-potential}
	Suppose that the potential $V$ is constant almost everywhere, 
	and that $A_1(x) = \ldots = A_N(x)$ for almost all $x \in \Omega$. 
	Then the following assertions are equivalent:
	\begin{enumerate}[\upshape (i)]
		\item 
		The matrices $\ue^{t V}$ converge on $\bbC^N$ as $t \to \infty$. 
		
		\item 
		The solutions to the coupled heat equation~\eqref{eq:coupled-heat-equation-on-bounded-domain-general} 
		converge uniformly as $t \to \infty$. 
	\end{enumerate}
\end{proposition}

\begin{proof}
	(i) $\Rightarrow$ (ii): 
	Since each of the (identical) semigroups $(\ue^{t \calA_k})_{t \geq 0}$ converges with respect to the operator norm as $t \to \infty$ (on $L^p(\Omega;\bbC)$ for any $p \in [1,\infty]$), so does the semigroup $\ue^{t \calB_2}$ (on $L^p(\Omega;\bbC^N)$ for any $p \in [1,\infty]$).
	Hence, this implication is an immediate consequence of Lemma~\ref{lemma:tensor-product-semigroups}. 

	(ii) $\Rightarrow$ (i): Suppose that $\ue^{t V}$ does not converge on $\bbC^N$ as $t \to \infty$. 
	Then there exists some vector $z \in \bbC^N$ such that $\ue^{t V} z$ does not converge as $t \to \infty$. 
	Now consider the function $\one \otimes \, z \coloneqq (z_1 \one, \ldots, z_N \one)^{\operatorname{T}}: \Omega \to \bbC^N$. 
	By Lemma~\ref{lemma:tensor-product-semigroups} one has
	\begin{align*}
		\ue^{t (\calB_2 + V)} (\one \otimes \, z) 
		=
		\ue^{tV} \ue^{t\calB_2} (\one \otimes \, z) = \ue^{tV} (\one \otimes \, z) = \one \otimes\, (\ue^{tV} z),
	\end{align*}
	for all $t \geq 0$ which does not converge as $t \to \infty$.
\end{proof}

As a simple example, we consider a similar potential as in Example~\ref{ex:toy-2-dissipative} -- and now we assume that the potential is constant, but we also allow for the case that it is $0$.

\begin{example} \label{ex:toy-constant}
	Let $N=2$, 
	let $a \in \bbR$, and let the potential $V$ be given by
	\begin{align*}
		V(x) \coloneqq V \coloneqq
		\begin{pmatrix}
			0 & -a \\
			a &  0
		\end{pmatrix} 
		\qquad \text{for all } x \in \Omega.
	\end{align*}
	As in Example~\ref{ex:toy-2-dissipative}, the $\ell^2$-dissipativity of the matrix $V$ implies that
	all solutions to the evolution equation~\eqref{eq:toy-example} 
	are bounded on $L^p(\Omega;\bbC^N)$ for any $p \in [1, \infty]$ (see Corollary~\ref{cor:ultracontractive-bounded}).
	
	The semigroup $(e^{tV})_{t \ge 0}$ on $\bbC^2$ converges as $t \to \infty$ if and only if $a = 0$.
	Hence, by Proposition~\ref{prop:constant-potential},
	we have convergence of all solutions to~\eqref{eq:toy-example} as $t \to \infty$ if and only if $a = 0$.
\end{example}

Of course, for this specific example, the same conclusion can also be derived from Theorem~\ref{thm:convergence-for-coupled-heat-equation-l2-case}.

\subsection{Simultaneously diagonalizable potentials} \label{subsection:decoupling}

In this section, we assume that
\begin{enumerate}
	\item[\upshape (a)] 
	one has $A_1(x) = \ldots A_N(x)$ for almost all $x \in \Omega$;
	
	\item[\upshape (b)] 
	the matrices $V(x)$ are simultaneously diagonalizable, 
	i.e.\ there exists an invertible matrix $U \in \bbC^{N \times N}$ as well as 
	bounded and measurable functions $\lambda_1, \ldots, \lambda_N: \Omega \to \bbC$ such that
	\begin{align}
		\label{eq:diagonalization-decoupling-system}
		U V(x) U^{-1} 
		= 
		\begin{pmatrix}
			 \lambda_1(x) &        &              \\
			              & \ddots &              \\
			              &        & \lambda_N(x)
		\end{pmatrix} 
		\qquad \text{for almost every } x \in \Omega.
	\end{align}
	
	\item[\upshape (c)] 
	we have $\re \lambda_k(x) \le 0$ for all $k \in \{1, \ldots, N\}$ and almost all $x \in \Omega$.
\end{enumerate}

As a result, the operators $\calA_1, \ldots, \calA_N$ coincide and, for the sake of notational simplicity, those operators will all be denoted by $\calA$. Assumptions~(a) and~(b) allow us to decouple the system \eqref{eq:coupled-heat-equation-on-bounded-domain-general} since $U \calB_2 U^{-1} = \calB_2$. This means that \eqref{eq:coupled-heat-equation-on-bounded-domain-general} is similar, via multiplication with $U$, to $N$ decoupled scalar equations
\begin{align} \label{eq:the-equation-diagonized}
	\frac{d}{dt}
	\begin{pmatrix}
		v_1 \\ \vdots \\ v_N
	\end{pmatrix}
	= 
	\begin{pmatrix}
		\calA v_1 \\
		\vdots \\
		\calA v_N
	\end{pmatrix}
	+ 
	\begin{pmatrix}
		\lambda_1 v_1 \\ \vdots \\ \lambda_N v_N
	\end{pmatrix} = 
	\begin{pmatrix}
		(\calA + \lambda_1) v_1 \\
		\vdots \\
		(\calA + \lambda_N) v_N
	\end{pmatrix}.
\end{align}
The solutions to \eqref{eq:cp-bounded} are then given by $u(t) = U^{-1} v(t)$ for all $t \geq 0$. 
Thus, the solutions to~\eqref{eq:cp-bounded} convergence uniformly if and only if the solutions to~\eqref{eq:the-equation-diagonized} do. On the other hand, the solutions to~\eqref{eq:the-equation-diagonized} converge uniformly if and only if this is true for the solutions to the scalar-valued equations 
\begin{align*}
	\frac{d}{dt} w = \calA w + \lambda_k w
\end{align*}
for all $k = 1, \ldots, N$ -- and by applying Theorem~\ref{thm:convergence-for-coupled-heat-equation-l2-case} to the scalar-valued case $N = 1$ (which is possible due to assumption~(c) on the eigenvalues $\lambda_k(x)$), we see that for any $k \in \{1, \ldots, N\}$ the solutions to the latter equation converge if and only if there is no $\ui\beta \in \ui\bbR \setminus \{0\}$ which is equal to $\lambda_k$ almost everywhere. Thus, one has the following result.

\begin{proposition} \label{prop:decoupled-system}
	Let the assumptions~{\upshape (a)--(c)} from the beginning of Subsection~\ref{subsection:decoupling} be satisfied.
	Then the following assertions are equivalent:
	\begin{enumerate}[\upshape (i)]
		\item The solutions to the coupled heat equation~\eqref{eq:coupled-heat-equation-on-bounded-domain-general} converge uniformly as $t \to \infty$.
		
		\item For each $k \in \{1, \ldots, N\}$ the following holds: 
		there does not exist a number $\ui \beta \in \ui \bbR \setminus \{0\}$ 
		which is equal to the function $\lambda_k$ almost everywhere.
	\end{enumerate}
\end{proposition}

We note that Example~\ref{ex:toy-2-dissipative} can also be treated by utilizing Proposition~\ref{prop:decoupled-system} since the matrices $V(x)$ in the example are simultaneously diagonalizable (and since the operators $\calA_1$ and $\calA_2$ in equation~\eqref{eq:toy-example} are both equal to the Laplace operator and thus coincide). 
Another simple example is the following:

\begin{example}
	\label{ex:toy-simultaneously-diagonalisable}
	Let $N=2$,
	let $a: \Omega \to \bbC$ be a bounded and measurable function 
	that satisfies $\re a(x) \ge 0$ for almost all $x \in \Omega$. 
	Let the potential $V$ be given by
	\begin{align*}
		V(x) \coloneqq 
		-a(x)
		\begin{pmatrix}
			1 & 2 \\
			1 & 2
		\end{pmatrix} 
		\qquad \text{for all } x \in \Omega.
	\end{align*}
	Then all the matrices $V(x)$ are simultaneously diagonalizable 
	and their eigenvalue curves are given by $\lambda_1(x) = 0$ and $\lambda_2(x) = -3a(x)$ for all $x \in \Omega$.
	
	So $\lambda_1$ is constant, but its value is not in $\ui \bbR \setminus \{0\}$. 
	Hence, Proposition~\ref{prop:decoupled-system} shows the following: 
	if $a$ is almost everywhere constant and equal to an element of $\ui \bbR \setminus \{0\}$, 
	then the solutions to~\eqref{eq:toy-example} do not converge as $t \to \infty$. 
	In all other cases the solutions to~\eqref{eq:toy-example} 
	converge uniformly as $t \to \infty$.
\end{example}

\subsubsection*{Concluding remarks}

Clearly, much more remains to be done in the case that the potential $V$ is not dissipative, since most methods presented in this paper do not work in this case. In fact, it is not even clear to the authors in general how to check boundedness of the solutions to~\eqref{eq:coupled-heat-equation-on-bounded-domain-general} if $V$ is not dissipative with respect to any $\ell^p$-norm on $\bbC^N$ (for one exception, though, see Proposition~\ref{prop:boundedness-quasipositive-potential}). 

Another direction of generalization is led by the idea to consider compact Riemannian manifolds in place of the bounded domain $\Omega \subseteq \bbR^d$.

Finally, in view of the coupled first order equations considered in \cite{Dobrick2023}, the question arises what happens, in general, if non-elliptic differential operators are coupled by a matrix-valued potential.

\subsubsection*{Acknowledgements}

We are indebted to the referee for pointing out how to prove Proposition~\ref{prop:eigenvalues-of-coupled-system} 
without the additional assumption that the diffusion coefficients $A_k(x)$ be symmetric.
We thank Abdelaziz Rhandi for pointing out the diagonalization argument outlined in Subsection~\ref{subsection:decoupling}. Furthermore, we are indebted to Fabian Wirth for bringing reference \cite{Belitskiui1988} to our attention. Moreover, Alexander Dobrick thanks Florian Pannasch for a fruitful discussion on results from \cite{Delmonte2011} regarding the theory of bi-continuous semigroups.

\appendix

\section{Dissipativity of real matrices} \label{section:dissipativity-of-real-matrices}

Dissipativity of matrices with real entries plays an important role in the main text. Therefore, we recall a characterization of dissipativity of matrices with respect to various $\ell^p$-norms in the following proposition.

\begin{proposition}	
	\label{prop:characterization-of-ell_p-dissipative-matrices}
	Let $N \in \mathbb{N}$ and let $C = (c_{jk}) \in \bbR^{N \times N}$. Then the following assertions hold:
	\begin{enumerate}[\upshape (a)]
		\item The matrix $C$ is dissipative with respect to the $\ell^2$-norm on $\bbR^N$ if and only if all eigenvalues of the symmetric part $\frac{1}{2}(C+C^T)$ of $C$ are contained in $(-\infty,0]$.
		
		\item The matrix $C$ is dissipative with respect to the $\ell^1$-norm on $\bbR^N$ if and only if
			\begin{align*}
				c_{kk} \le - \sum_{j \in \{1,\ldots,N\} \setminus \{k\}} \modulus{c_{jk}}
			\end{align*}
			for each $k \in \{1,\ldots,N\}$.
			
		\item The matrix $C$ is dissipative with respect to the $\ell^\infty$-norm on $\bbR^N$ if and only if
			\begin{align*}
				c_{kk} \le - \sum_{j \in \{1,\ldots,N\} \setminus \{k\}} \modulus{c_{kj}}
			\end{align*}
			for each $k \in \{1,\ldots,N\}$.
			
		\item Let $p \in [1,\infty)$. The matrix $C$ is dissipative with respect to the $\ell^p$-norm on $\bbR^N$ if and only if
			\begin{align*}
				(\sgn \xi \cdot \modulus{\xi}^{p-1})^{\operatorname{T}}\, C \xi \le 0
			\end{align*}
			for all $\xi \in \bbR^N$; here, the vector $\sgn \xi \in \{-1,0,1\}^N$ contains the signs of the entries of $\xi$, and its product with the vector $\modulus{\xi}^{p-1}$ is computed entrywise.
	\end{enumerate}
\end{proposition}
\begin{proof}
	(a) It follows from $\xi^{\operatorname{T}} C \xi = \xi^{\operatorname{T}} \, \frac{1}{2}(C^{\operatorname{T}}+C) \xi$ for all $\xi \in \bbR^N$ that $C$ is dissipative with respect to the $\ell^2$-norm on $\bbR^N$ if and only if $\frac{1}{2}(C+C^{\operatorname{T}})$ is so. Since $\frac{1}{2}(C+C^{\operatorname{T}})$ is symmetric, this proves the assertion.
	
	(d) Endow $\bbR^N$ with the $\ell^p$-norm for a fixed $p \in [1,\infty)$. For each $\xi \in \bbR^N$ of norm $\norm{\xi}_p = 1$ the vector $\sgn \xi \cdot \modulus{\xi}^{p-1}$, if considered as an element of the dual space of $\bbR^N$, also has norm equal to $1$ and satisfies $(\sgn \xi \cdot \modulus{\xi}^{p-1})^{\operatorname{T}} \xi = 1$. Since a matrix is dissipative if and only if it is strictly dissipative, this proves~(d).
	
	(b) If the estimate in~(ii) is satisfied, then it follows for each $\xi \in \bbR^N$ that
	\begin{align*}
		(\sgn \xi)^{\operatorname{T}} C \xi & = \sum_{j=1}^N \sum_{k=1}^N \sgn(\xi_j) \, c_{jk} \, \xi_k \\
		& = \sum_{k=1}^N \left(c_{kk} \modulus{\xi_k} \; + \sum_{j \in \{1,\ldots,N\} \setminus \{k\}}^N \sgn(\xi_j) c_{jk}\xi_k \right) \\
		& \le \sum_{k=1}^N \left(c_{kk} \modulus{\xi_k} \; + \sum_{j \in \{1,\ldots,N\} \setminus \{k\}}^N  \modulus{c_{jk}} \modulus{\xi_k} \right) \le 0,
	\end{align*}
	so $C$ is dissipative according to~(d). 
	
	Now assume conversely that $C$ is dissipative. Fix $k \in \{1,\ldots,N\}$ and let $e_k \in \bbR^N$ denote the $k$-th canonical unit vector. Moreover, we define a vector $\xi \in \bbR^N$, which we consider as a functional on $\bbR^N$, in the following way: we set $\xi_k = 1$ and $\xi_j = \sgn c_{jk}$ for all $j \in \{1,\ldots,N\} \setminus \{k\}$. Then $\norm{\xi}_\infty = 1$ and $\xi^T e_k = 1$, so the strict dissipativity of $C$ implies that
	\begin{align*}
		0 \ge \xi^{\operatorname{T}} C e_k = \sum_{j=1}^N \xi_j c_{jk} = c_{kk} \; + \sum_{j \in \{1,\ldots,N\} \setminus \{k\}} \modulus{c_{jk}},
	\end{align*}
	which proves the assertion.
	
	(c) This follows from~(b) by duality.
\end{proof}

In the above proof we could, of course, also first prove~(c) directly and then derive~(b) from~(c) by duality; this is for instance done in \cite[Remark~2.1.2]{GlueckDISS}.

The above characterizations can be adapted to the case of complex matrices, of course.
For instance, a matrix $C \in \bbC^{N \times N}$ is dissipative with respect to the $\ell^2$-norm on $\bbC^N$ if and only if all eigenvalues of the Hermitian part $\frac{1}{2}(C+C^*)$ of $C$ are located in $(-\infty,0]$.
We refrain from explicitly stating the complex case of~(ii)--(iv) since $\ell^p$-dissipativity of matrices is, 
for $p \not=2$, only used for real matrices in the main text.

\bibliographystyle{plain}
\bibliography{literature}

\end{document}